\documentclass{amsart}
\usepackage[dvipsnames,table]{xcolor}
\definecolor{DarkPurple}{RGB}{102,0,153}
\usepackage[
    colorlinks=true,
    linkcolor=OliveGreen,   
    citecolor=DarkPurple, 
    urlcolor=NavyBlue
]{hyperref}
\usepackage{amsmath}
\usepackage{amsthm}
\usepackage{amssymb}
\usepackage{tikz-cd}
\usepackage{etoolbox}
\usepackage{blkarray}
\usepackage{enumitem}
\usepackage{ytableau}
\usepackage{graphicx}
\usepackage{pgfplots}
\pgfplotsset{compat=newest}
\numberwithin{equation}{section}

\usepackage{tikz}
\usetikzlibrary{calc}

\newtheorem{thm}{Theorem}[section]
\newtheorem{theorem}[thm]{Theorem}
\newtheorem{corollary}[thm]{Corollary}

\newtheorem{lemma}[thm]{Lemma}
\newtheorem{prop}[thm]{Proposition}
\newtheorem{proposition}[thm]{Proposition}

\theoremstyle{definition}
\newtheorem{definition}[thm]{Definition}
\newtheorem{example}[thm]{Example}
\newtheorem{remark}[thm]{Remark}

\title[Markov chains on Weyl groups from the geometry of $G/B$]{Markov chains on Weyl groups from the geometry of the flag variety}
\thanks{\textbf{MSC 2020}: Primary: 60J10, 14M15. Secondary: 05E10, 05A17, 20G40.}

\author{Persi Diaconis}
\author{Calder Morton-Ferguson}

\begin{document}

\begin{abstract}
    This paper studies a basic Markov chain, the Burnside process, on the space of flags $G/B$ with $G = \mathrm{GL}_n(\mathbb{F}_q)$ and $B$ its upper triangular matrices. This gives rise to a shuffling: a Markov chain on the symmetric group realized via the Bruhat decomposition $S_n \cong B \backslash G/B$. Actually running and describing this Markov chain requires understanding Springer fibers and the Steinberg variety. The main results give a practical algorithm for all $n$ and $q$ and determine the limiting behavior of the chain when $q$ is large. In describing this behavior, we find interesting connections to the combinatorics of the Robinson--Schensted correspondence and to the geometry of orbital varieties. The construction and description is then carried over to finite Chevalley groups of arbitrary type, describing a new class of Markov chains on Weyl groups.
\end{abstract}

\maketitle
\setcounter{tocdepth}{1}
\tableofcontents

\section{Introduction}
A basic group theory question is to understand the decomposition of a group $G$ into double cosets $H \backslash G / K$. For $G$ a finite group and $H,K$ subgroups, one may ask:
\begin{itemize}
    \item How many double cosets are there?
    \item What are typical sizes?
    \item Do the double cosets have ``nice names"?
    \item Do they fit together into some sort of understandable moduli space?
\end{itemize}

A famous example is the Bruhat decomposition
\[
   W = U\backslash G/B,
\]
with $G = \mathrm{GL}_n(\mathbb{F}_q)$ the invertible matrices over a finite field $\mathbb{F}_q$, $B$ the upper triangular matrices, $U$ the unipotent matrices in $B$, and $W = S_n$. This decomposition is realized by Gaussian elimination, and generalizes to the case of $G$ a finite Chevalley group and $W$ its Weyl group. Going back to the case of a general finite group $G$, the Burnside process is a stochastic algorithm for picking a double coset at random. This gives an effective way to get approximate answers to the first two questions above.

This paper began with the practical question: how would one actually carry out the Burnside process for the action of $U$ on $G/B$, for $G = \mathrm{GL}(\mathbb{F}_q)$? In this case $G/B$ is the space of flags $F$ of vector spaces in $\mathbb{F}_q^n$. The process entails two steps:

\begin{enumerate}
    \item From some $F \in G/B$, choose $u \in U$ with $F = uF$, uniformly at random.
    \item From this $u \in U$, choose $F' \in G/B$ with $F' = uF'$, uniformly at random.
\end{enumerate}

The chain moves from $F$ to $F'$ with probability given by
\begin{align}\label{eqn:pwprobintro}
    P(F, F') & = \frac{1}{|\mathrm{stab}_U(F)|} \sum_{g \in \mathrm{stab}_U(F) \cap \mathrm{stab}_U(F')} \frac{1}{|\mathrm{Fix}_X(g)|},
\end{align}
where Section~\ref{sec:burnside} explains the notation. This gives a Markov chain on $G/B$ with stationary distribution
   \begin{align*}
    \pi(x) = \frac{1}{n!|O_F|},
\end{align*}
where $O_F$ is the $U$-orbit of $F$. Thus simply reporting the orbit after each step gives a method for making a uniform choice of orbit. The orbits are indexed by permutations, so a random choice of $w \in S_n$ results.

It turns out that even in this special case, practical implementation of steps (1)--(2) requires an understanding of the geometry and combinatorics of flag varieties and Springer fibers. Fixing $u \in U$, the set
\[\{F : F^u = F\}\]
is the Springer fiber: a variety whose cohomology underlies the representation theory of $W$. Running the Burnside process involves understanding the size of Springer fibers, while computing explicit transition probabilities involves understanding the intersection of Springer fibers with a fixed double coset $U w B$. This is a difficult task in general. For $\mathrm{GL}_n$, we are able to give an explicit implementation of the Burnside process using connections between Springer fibers and the combinatorics of Hall--Littlewood and Green polynomials.

Section~\ref{sec:burnside} gives background on the Burnside process and the geometry of flag space required to carry out steps (1) and (2) for $G = \mathrm{GL}_n(\mathbb{F}_q)$. Section~\ref{sec:sampling} spells out the details of a full sampling algorithm for this case. This algorithm has been implemented, and actual runs of the Markov chain are presented and discussed in Section~\ref{sec:sim}. This may be read now for further motivation.

Hands-on understanding of the Markov chain, even for $W = S_n$, is difficult. However, when $q$ is large and $n$ fixed, we show that there is a simple description of its behavior in terms of Young tableaux and the Robinson--Schensted correspondence. This is developed in Section~\ref{sec:typea} with examples in Section~\ref{sec:sim}. This analysis gives a useful lower bound on the convergence of the Markov chain to its stationary distribution whose details are in Section~\ref{sec:mixing}. We also discuss a more precise analysis and set of bounds in the case of $G = \mathrm{GL}_3(\mathbb{F}_q)$ in Proposition \ref{prop:gl3precise}.

Section~\ref{sec:generaltype} develops the same theme for the more general case of finite Chevalley groups in any type. It has a self-contained introduction and analogues of the main results of Section~\ref{sec:typea}. Now the Young tableaux and their associated combinatorics must be replaced by more sophisticated tools related to Steinberg and orbital varieties. A practical version of our algorithm from Section~\ref{sec:sampling} in this general case remains for the future. That being said, in all types we are still able to describe large $q$ behavior in terms of interesting combinatorial data of general Weyl groups called Steinberg cells.

The topics developed may interest researchers in disparate fields (probability, combinatorics, geometric representation theory). Thus we have tried to write a broadly accessible account.

\subsection*{Acknowledgments} We thank Michael Howes, Matthew Nicoletti, and Arun Ram for very helpful comments. P.D. was supported in part by NSF grant 1954042.

\section{The Burnside process on \texorpdfstring{$G/B$}{G/B}}\label{sec:burnside}

\subsection{The Burnside process for the action of a finite group}

\subsubsection{General setup}

The Burnside process for the action of a finite group $G$ on a set $X$ is a Markov chain on $X$ built from the following two steps. From a given element $x \in X$ one proceeds as follows:
\begin{enumerate}
    \item Choose an element $g \in \mathrm{stab}_G(x)$ uniformly at random.
    \item Choose an element $y \in \mathrm{Fix}_X(g)$ uniformly at random.
\end{enumerate}
Here $\mathrm{stab}_G(x) \subset G$ is the stabilizer of $x$ and $\mathrm{Fix}_X(g) \subset X$ is the set of elements fixed by $g$. The transition probability $P(x, y)$ of moving from $x$ to $y$ is then given by the formula in (\ref{eqn:pwprobintro}). It is explained in \cite{DH} that $P$ is an ergodic, reversible Markov kernel with stationary distribution given by
\begin{align*}
    \pi(x) = \frac{|\Theta|^{-1}}{|O_x|},
\end{align*}
where $\Theta$ is the set of $G$-orbits on $X$ and $O_x$ is the orbit $G\cdot x$.

One can then consider the Burnside process on the set $\Theta$ of $G$-orbits, obtained by remembering only the $G$-orbit at each step of the Burnside process. One then gets a reversible Markov chain on orbits whose stationary distribution is always uniform; we call its Markov kernel $P_\Theta$.

The Burnside process has been used to generate and count ``P{\'o}lya
trees" (shapes of rooted labeled trees) on $n$ vertices for $n$ of order
$10^8$ \cite{BD}, to generate random partitions of $n$ for $n$ of order $10^8$ and
random contingency tables ($i \times j$ arrays of non-negative integers with
given row and column sums) \cite{DH} and for counting the number of conjugacy
classes of the uni-upper triangular matrices in $\mathrm{GL}_n(\mathbb{F}_q)$ \cite{DZ21}. In all of these applications, as in the present paper, there was substantial mathematical and computational effort required to get the two steps up and running.  Convergence estimates such as a spectral gap for the
underlying Markov chain were left for future work. There have been successful running time estimates for the Burnside
process in special cases; see \cite{DLR}, \cite{Pag23}, \cite{Rah20}. We hope to do this for the walks studied here.

\subsubsection{The Burnside process on the flag variety for $\mathrm{GL}_n(\mathbb{F}_q)$} Now let $G = \mathrm{GL}_n(\mathbb{F}_q)$, invertible $n\times n$ matrices with entries in the finite field $\mathbb{F}_q$. Let $B$ be its subgroup of upper triangular matrices and $U$ the set of unipotent matrices in $B$. We then let $X = G/B$; this is the \emph{flag variety} of $\mathrm{GL}_n(\mathbb{F}_q)$ over the finite field $\mathbb{F}_q$. It is the set of \emph{complete flags}
\[0 = V_0 \subset  V_1 \subset \dots \subset V_n = \mathbb{F}_q^n\]
of vector subspaces of $\mathbb{F}_q^n$ where $\dim(V_k) = k$ for all $1 \leq k \leq n$; we write $V_\bullet$ for convenience. It admits a left action of the group $U$. We now recall the Bruhat decomposition, which describes the orbits of the action of $U$ (or equivalently, of $B$) on $X$. It gives that the orbits of this action are labelled by elements of the symmetric group $S_n$. We write $\ell$ for the length function on $S_n$, which counts the number of inversions of a permutation.
\begin{prop}\label{prop:bruhat}
    There is a decomposition
    \begin{align*}
        G/B & = \coprod_{w \in S_n} BwB = \coprod_{w \in S_n} UwB,
    \end{align*}
    where for each $w \in S_n$, $UwB = BwB$ is considered as a subset of $G/B$, and $|UwB| = q^{\ell(w)}$.
\end{prop}

In the present paper we study the Burnside process for the action of $U$ on $X$. Given $x, y \in X$, we will write $P(x, y)$ for the transition probabilities, and since by Proposition \ref{prop:bruhat} the orbits for this action are indexed by $S_n$, for $w, z \in W$ we will use $P_{S_n}(w, z)$ to denote the transition probabilities for the corresponding Burnside process on the set of orbits. We refer to it as the $q$-Burnside process when emphasizing the dependence on $q$.

\subsection{Stabilizers in \texorpdfstring{$U$}{U}} Given some choice of $x \in X$, the first step of the Burnside process requires uniform sampling from $\mathrm{stab}_U(x)$. Using Proposition \ref{prop:bruhat} to write $x = uwB$ for some $u \in U$, $w \in W$, we note that
\begin{align*}
    \mathrm{stab}_U(x) & = u \cdot \mathrm{stab}_U(wB) \cdot u^{-1},
\end{align*}
so we now recall an explicit description of $\mathrm{stab}_U(wB)$ for any $w \in W$. 

\begin{definition}
    For any $w \in W$, let
    \[
    U_w = \mathrm{stab}_U(wB) \subset U.
    \]
    Setting $U^w = wUw^{-1}$, one then clearly has $U_w = U \cap wUw^{-1} = U \cap U^w$. 
\end{definition}

\begin{proposition}[see Proposition \ref{prop:uw}]\label{prop:glnuw}
    For any $w \in S_n$, let \[\mathrm{Inv}(w) = \{(i,j) \mid i<j, \; w(i) > w(j)\}\] denote the set of inversions of $w$. 
    
    Then
    \[
        U_w = \left\{ (u_{ij}) \in U \;\middle|\; 
        u_{ij} = 0 \text{ for } (i,j) \in \mathrm{Inv}(w) \right\}.
    \]
\end{proposition}
In other words, $U_w$ consists of upper triangular unipotent matrices in which the entries corresponding to inversions of $w$ are forced to vanish. In particular, for any $w \in W$ the cardinality of the indexing set in the product which appears in Proposition \ref{prop:glnuw} is exactly $\binom{n}{2} - \ell(w)$, so
\begin{equation}\label{eqn:uw}
    |U_w| = q^{\binom{n}{2} - \ell(w)}.
\end{equation}

\begin{example}
    Suppose $n = 3$. In this case, $W = \{e, s_1, s_2, s_2s_1, s_1s_2, w_0\} \cong S_3$, where $s_1 = (12), s_2 = (23)$ and $w_0 = s_1s_2s_1$. 

    \begin{align*}
        \mathrm{stab}_{U}(s_1B) & \cong \left\{\begin{pmatrix}
            1 & 0 & *\\
            0 & 1 & *\\
            0 & 0 & 1
        \end{pmatrix}\right\} & 
        \mathrm{stab}_{U}(s_2B) & \cong \left\{\begin{pmatrix}
            1 & * & *\\
            0 & 1 & 0\\
            0 & 0 & 1
        \end{pmatrix}\right\},\\
        \mathrm{stab}_{U}(s_2s_1B) & \cong \left\{\begin{pmatrix}
            1 & * & 0\\
            0 & 1 & 0\\
            0 & 0 & 1
        \end{pmatrix}\right\} & 
        \mathrm{stab}_{U}(s_1s_2B) & \cong \left\{\begin{pmatrix}
            1 & 0 & 0\\
            0 & 1 & *\\
            0 & 0 & 1
        \end{pmatrix}\right\}
    \end{align*}
    while $\mathrm{stab}_{U}(eB) \cong U$ and $\mathrm{stab}_{U}(w_0B) \cong \{1\}$.
\end{example}

\subsection{Fixed-point sets and Springer fibers}

Now given an arbitrary $u \in U$, the second step in the Burnside process calls for an understanding of $\mathrm{Fix}_{X}(u)$, which we refer to as $X_u$. This is the set of flags $V_\bullet$ such that
\begin{align*}
    u(V_k) \subset V_k
\end{align*}
for all $1 \leq k \leq n$.

This is the \emph{Springer fiber} associated to the unipotent matrix $u$. We will explain the relationship between the geometry of $X_u$ to the transition probabilities in the Burnside process on $S_n$. A starting point for a discussion of these combinatorics is an understanding of the cardinality $|X_u|$ itself.

\subsubsection{Green polynomials}

In \cite[III.3]{Macdonald}, the polynomials
\[
    X_\mu^\lambda(q)
\]
associated to partitions $\mu, \lambda$ of $n$ are defined as the coefficients expressing the power-sum symmetric functions in terms of Hall–Littlewood polynomials:
\begin{align*}
    p_\mu(x) \;=\; \sum_{\lambda \vdash n} X_\mu^\lambda(q) P_\lambda(x;q).
\end{align*}
Here $\mu$ and $\lambda$ are partitions of $n$, $p_\mu$ is the power-sum symmetric function, and $P_\lambda(x;q)$ is the Hall–Littlewood polynomial. Thus the $X_\mu^\lambda(q)$ form the entries of the transition matrix from the power-sum basis $\{p_\mu\}$ to the Hall–Littlewood basis $\{P_\lambda\}$.

The \emph{Green polynomials} $Q_\mu^\lambda(q)$ are related to these coefficients by the reciprocity identity
\begin{align}\label{eqn:recipxq}
    q^{-b(\lambda)} Q_\mu^\lambda(q) \;=\; X_\mu^\lambda\!\left(1/q\right),
\end{align}
where 
\[
    b(\lambda) \;=\; \sum_{i \geq 1} (i-1)\lambda_i = \deg(X_\mu^\lambda) = \deg(Q_\mu^\lambda).
\]
We note that in \cite{Macdonald} it is referred to as $n(\lambda)$, a notation we avoid here so as not to confuse it with the $n$ in $S_n$.

We are primarily interested in the case $\mu = (1^n)$; in this case, the polynomials
\[
    Q_{(1^n)}^\lambda(q)
\]
encode the cardinalities of the Springer fibers $X_u$ as follows.

\begin{proposition}[see {\cite[Ch.~III, §7]{Macdonald}}]\label{prop:pointcount}
    Let $u \in \mathrm{GL}_n(\mathbb{F}_q)$ be a unipotent matrix with Jordan type $\lambda \vdash n$. Then
    \[
        |X_u| \;=\; Q_{(1^n)}^\lambda(q),
    \]
    where $Q_{(1^n)}^\lambda(q)$ is the corresponding Green polynomial.
\end{proposition}

Although here we consider $X_u$ only as a finite set, the Springer fiber $\mathbf{X}_u$ can be similarly defined and considered as a variety over $\overline{\mathbb{F}}_q$. We note that by Proposition \ref{prop:pointcount}, $X_{(1^n)}^\lambda(0)$, which is the leading coefficient of $Q_{(1^n)}^\lambda(q)$, can be viewed as the number of irreducible components of $\mathbf{X}_u$. In \cite{Spalt}, it is explained that when $u \in \mathrm{GL}_n(\mathbb{F}_q)$ has Jordan type $\lambda \vdash n$, these irreducible components are in bijection with the set of standard Young tableaux of shape $\lambda$. One can also directly deduce that $X_{(1^n)}^\lambda(0)$ is the number of such tableaux using the following combinatorial argument.
\begin{prop}\label{prop:flambda}
    For any partition $\lambda$ of $n$,
    \begin{align*}
        X_{(1^n)}^\lambda(0) & = f^\lambda,
    \end{align*}
    where $f^\lambda$ is the number of standard Young tableaux of shape $\lambda$.
\end{prop}
\begin{proof}
    By definition, $X_{(1^n)}^\lambda(q)$ is a coefficient of the transition matrix between power-sum products and Hall-Littlewood functions. This means
    \begin{align*}
        p_{(1^n)}(x) & = \sum_\mu X_{(1^n)}^\mu(q) P_\mu(x; q).
    \end{align*}
    We note that $p_{(1^n)}(x) = p_1(x)^n$, where $p_1(x) = x_1 + \dots + x_n$, while $P_\mu(x; 0) = s_\mu(x)$, the Schur function corresponding to $\mu$.

    Noting that $p_1(x)$ is the Frobenius characteristic of the regular representation of $S_n$, we have
    \begin{align*}
        p_1(x)^n & = \sum_\mu f^\mu s_\mu(x),
    \end{align*}
    since $f^\mu$ is the dimension of the irreducible $S_n$-representation indexed by $\mu$. Comparing coefficients on the Schur polynomial $s_\lambda(x)$ gives $X_{(1^n)}^\lambda(0) = f^\lambda$.
\end{proof}

\subsubsection{Springer fibers and Bruhat cells}\label{sec:wild}

We now explain why an understanding of transition probabilities for our Burnside process reduces to an understanding of the cardinalities of the sets
\begin{align}\label{eqn:hard}
    X_u \cap BwB
\end{align}
for all $w \in W$.

The following comes from combining formulas (\ref{eqn:pwprobintro}) and (\ref{eqn:uw}). 
\begin{proposition}
The transition probability $P_{S_n}(w, z)$ for the $q$-Burnside process on $S_n$ is given by
    \begin{align*}
        P_{S_n}(w, z) & = \frac{1}{q^{\binom{n}{2} - \ell(w)}}\sum_{u \in U_{w}(\mathbb{F}_q)} \frac{1}{Q_{(1^n)}^\lambda(q)} |X_u \cap ByB|
    \end{align*}
\end{proposition}

This means that an understanding of the quantities $|X_u \cap ByB|$
for all $u \in U$ and $y \in S_n$ would yield a precise understanding of the transition probabilities in our Markov matrix $P_{S_n}$. In this section, we give a brief overview of what is known of these quantities. The upshot is that for certain ``carefully-chosen" $u$, there is a very nice combinatorial description of $|X_u \cap ByB|$ in terms of $y$ and the Jordan type of $u$, whereas determining the cardinalities $|X_u \cap ByB|$ for arbitrary $u$ has been shown to be very difficult. (More precisely, \cite[Section 6]{T} indicates that necessary and sufficient conditions on $u$ for the following combinatorial description to hold are not known, but a sufficient condition is explained in detail in \cite[Section 4]{T2}.)

It is always true (for any $u$) that $\mathbf{X}_u$ has an \emph{affine paving} (a filtration by closed subvarieties whose successive differences are isomorphic to affine spaces) which can be used to give a geometric explanation for an expression of $|X_u|$ as a positive linear combination of monomials $q^d$. If $u$ is ``carefully chosen" in the sense of \cite{T2}, then the resulting partition of $X_u$ is given by $\{X_u \cap ByB\}_{y \in W}$, and therefore each quantity $|X_u \cap ByB|$ is equal to $q^d$ for $d$ the dimension of the corresponding affine space. The dimension $d$ can then be computed combinatorially in terms of $y$ and the Jordan type of $u$ in terms of the number of \emph{row-strict fillings} of a certain Young diagram; this combinatorial description is explained in \cite[Section 3.2]{T}

It is explained in \cite{T} that little is known about the intersections $|X_u \cap ByB|$ for other $u$ other than the fact that they can be very complicated; examples where these intersections deviate from the combinatorial formula previously explained are exhibited in \cite{Kos} and \cite{Rie}. It is known (c.f.\ \cite{T}) to be a difficult open problem to produce a formula for these quantities in more generality.

We view this as evidence toward our belief that the precise transition probabilities which appear in our Burnside process have the potential to be extremely complicated, and we do not believe it is likely that they can be precisely characterized in any neat algebraic form. Despite this, we are able to analyze leading order terms in $q$ of these transition probabilities to develop a complete understanding of the approximate behavior of the Markov chain up to ``error terms" of order $O(1/q)$. Such a characterization is the main result of Sections \ref{sec:typea} and \ref{sec:generaltype}.

\begin{remark}
    Despite how complicated we expect general entries of the transition matrix $\{P_{S_n}(w, z)\}_{w, z\in S_n}$ to be, we note that the single column $\{P_{S_n}(w_0, w)\}_{w \in S_n}$ (where $w_0$ is the unique ``longest permutation" with $\binom{n}{2}$ inversions) is very simple. In fact, for any $w \in S_n$,
    \begin{align*}
        P_{S_n}(w_0, w) & = \frac{q^{\ell(w)}}{|X|}
    \end{align*}
    This means that starting at $w_0$ and running the Burnside process for one step is equivalent to sampling from the \emph{Mallows measure}, a well-studied probability measure on $S_n$. This is because $U_{w_0}$ is trivial, and therefore this step of the Burnside process samples uniformly from $X$. The Bruhat cell corresponding to any $w \in S_n$ has cardinality $q^{\ell(w)}$, from which this observation follows.
\end{remark}

\section{A sampling algorithm}\label{sec:sampling}

In this section, we continue with the case of $G = \mathrm{GL}_n(\mathbb{F}_q)$. In this case, the Burnside process gives rise to a \emph{shuffling}, i.e.\ a Markov chain on $S_n$. In the present section, we describe an explicit sampling algorithm which one can use to implement the shuffling arising from this Burnside process in practice. This algorithm can be implemented with code (we implemented it in SageMath to run the simulations explained in Section \ref{sec:sim}) and runs efficiently even in cases when a brute-force enumeration of flags is impractical.

\subsection{Some partition combinatorics}

Suppose that $\lambda = (\lambda_1, \dots, \lambda_k)$ is a partition of $n$; we freely identify it with its corresponding Young diagram. When describing such a Young diagram (and in the examples in Section \ref{sec:sim}) we will always use English notation, so the row lengths are weakly decreasing when reading from top to bottom. We can then make the following definitions.

\begin{definition}
    Let $R(\lambda)$ be the set of partitions corresponding to those Young diagrams obtained by removing a single box from the Young diagram of $\lambda$. In other words, partitions which arise from decrementing a single part of $\lambda$ by $1$.

    Given $\lambda' \in \lambda$, we write $r(\lambda, \lambda')$ to denote the index of the part of $\lambda$ which was decremented by $1$ to form $\lambda'$. By convention, $r(\lambda, \lambda')$ is always maximal among those $j$ for which $\lambda_j = \lambda_{r(\lambda, \lambda')}$. 
    
    Finally, let $m(\lambda, \lambda')$ denote the multiplicity of $\lambda_{r(\lambda, \lambda')}$ in $\lambda$.
\end{definition}

\begin{definition}
    Let $I(\lambda) \subset \{1, \dots, n\}$ be the $k$-element set consisting of $1$ along with all integers of the form
    \begin{align}
        \lambda_1 + \lambda_2 + \dots + \lambda_i + 1
    \end{align}
    for $1 \leq i \leq k-1$.
\end{definition}

For any partition $\lambda$ of $n$, we write $J_\lambda$ for the unipotent matrix in Jordan canonical form with its ordered set of $k$ Jordan blocks having sizes $(\lambda_1, \dots, \lambda_k)$.

The following is immediate.
\begin{lemma}
    The $k$-dimensional subspace of $\mathbb{F}_q^n$ consisting of eigenvectors of $J_\lambda$ is
    \begin{align}
        E(\lambda) = \mathrm{span}\{e_{i}\}_{i \in I(\lambda)}.
    \end{align}
\end{lemma}

If $v \in E(\lambda)$, then the linear transformation $\overline{J_\lambda}$ is well-defined on the quotient $\mathbb{F}_q^n/\mathrm{span}(v)$. 
\begin{definition}
    For $\lambda' \in R(\lambda)$, let $E^{\lambda'}(\lambda)$ be the subset of $E(\lambda)$ consisting of eigenvectors $v$ of $J_\lambda$ for which the transformation $\overline{J}_\lambda$ has Jordan type $\lambda'$ on $\mathbb{F}_q^n/\mathrm{span}(v)$
\end{definition}

\begin{lemma}\label{lem:elambda}
    Suppose $\lambda' \in R(\lambda)$. Then $E^{\lambda'}(\lambda)$ is exactly the set of $(v_i)_{i=1}^n$ such that
    \begin{itemize}
        \item $v_i = 0$ unless $i \in I(\lambda)$,
        \item $v_i = 0$ for $i \geq \lambda_1 + \dots + \lambda_{r(\lambda, \lambda')} + 1$,
        \item $v_i \neq 0$ for at least one $i$ satisfying
        \[\lambda_1 + \dots + \lambda_{r(\lambda, \lambda') - m(\lambda, \lambda')} + 1 \leq i < \lambda_1 + \dots + \lambda_{r(\lambda, \lambda')} + 1.\]
    \end{itemize}
\end{lemma}
This gives an explicit bijection
\begin{align}
    E^{\lambda'}(\lambda) \cong \mathbb{F}_q^{r(\lambda, \lambda') - m(\lambda, \lambda')} \times (\mathbb{F}_q^{m(\lambda, \lambda')} \setminus \{0\})
\end{align}
for any $\lambda' \in R(\lambda)$, which leads to the following result.

\begin{corollary}\label{cor:sizeofe}
    Suppose $\lambda' \in R(\lambda)$. Then
    \begin{align}
        |E^{\lambda'}(\lambda)| & = q^{r(\lambda, \lambda') - m(\lambda, \lambda')}(q^{m(\lambda, \lambda')} - 1) = q^{r(\lambda, \lambda')} - q^{r(\lambda, \lambda') - m(\lambda, \lambda')}
    \end{align}
\end{corollary}

\subsection{The sampling algorithm}

Suppose that $V_\bullet = 0 \subset V_1 \subset \dots \subset V_n = \mathbb{F}_q^n$ is a complete $n$-step flag of subspaces in $\mathbb{F}_q^n$. We now describe an explicit algorithm for sampling another flag of subspaces according to the Burnside process for the action of $U$ on $G/B$.

\subsubsection{Bruhat decomposition via Gaussian elimination} \label{sec:bruhatgaussian}
Choose a basis $\{v_i\}_{i=1}^n$ for $\mathbb{F}_q^n$ such that for all $1 \leq k \leq n$,
\begin{align}
    V_k & = \mathrm{span}(v_1, \dots, v_k).
\end{align}
Then $V_\bullet$ is the flag corresponding to the point $gB \in G/B$, where $g \in G$ is the matrix with $v_i$ (written with respect to the standard basis $\{e_j\}_{j=1}^n$) as its $i$th column. This identification is independent of the choice of basis.

Bruhat decomposition guarantees that one can then apply Gaussian elimination (represented as multiplication by a unipotent upper triangular matrix on the left) to reduce $g$ to a permutation matrix. This gives an element $u \in U$ and $w \in S_n$ such that
\begin{align*}
    uwB = gB \in G/B.
\end{align*}
A very clear exposition of Bruhat decomposition as Gaussian elimination appears in Roger Howe's expository work \cite{Howe}.

\subsubsection{Sampling from the stabilizer}
First we explain that uniformly sampling from $\mathrm{stab}_U(uwB) \subset U$ is straightforward. Indeed, note that
\begin{align}
    \mathrm{stab}_U(uwB) = u\cdot \mathrm{stab}_U(w)\cdot u^{-1} = uU_wu^{-1},
\end{align}
so it is enough to sample uniformly from $U_w$ and then conjugate the resulting element by $u$.

Since by Proposition \ref{prop:glnuw},
\begin{align}
    U_w = \left\{ (u_{ij}) \in U \;\middle|\; 
        u_{ij} = 0 \text{ for } (i,j) \in \mathrm{Inv}(w) \right\} \cong \mathbb{F}_q^{\ell(w_0) - \ell(w)},\label{eqn:uwsample}
\end{align}
it is enough to sample uniformly from $\mathbb{F}_q^{\ell(w_0) - \ell(w)}$. Passing the result through the bijection (\ref{eqn:uwsample}) and conjugating by $u$, we have chosen an element $a \in \mathrm{stab}_U(uwB)$ uniformly at random. We then write $a = QJQ^{-1}$ for some $Q \in G$ with $J$ a unipotent upper triangular matrix in Jordan form with Jordan blocks of non-increasing size corresponding to a permutation $\lambda = (\lambda_1, \dots, \lambda_k)$.

\subsubsection{Weighted sampling on $R(\lambda)$}
The next step of our algorithm is a weighted sampling among elements of $R(\lambda)$. For any $\lambda' \in R(\lambda)$, let
\begin{align*}
    \mathrm{wt}(\lambda') = (q^{r(\lambda, \lambda')} - q^{r(\lambda, \lambda') - m(\lambda, \lambda')})Q_{(1^n)}^{\lambda'}(q).
\end{align*}
We then perform a weighted sampling among elements $\lambda' \in R(\lambda)$ with probabilities proportional to the weights $\mathrm{wt}(\lambda')$. We write $\mu$ for the resulting element of $R(\lambda)$.

\subsubsection{Sampling from the Springer fiber}

By means of the bijection 
\begin{align*}
    E^{\mu}(\lambda) \cong \mathbb{F}_q^{r(\lambda, \mu) - m(\lambda, \mu)} \times (\mathbb{F}_q^{m(\lambda, \mu)} \setminus \{0\})
\end{align*}
described in Lemma \ref{lem:elambda}, one then samples an element of $E^{\mu}(\lambda)$ uniformly at random, calling it $v$.

Since by Corollary \ref{cor:sizeofe} there are $q^{r(\lambda, \mu)} - q^{r(\lambda, \mu) - m(\lambda, \mu)}$ vectors in $E^\mu(\lambda)$. This means that these two steps (weighted sampling among elements of $R(\lambda)$ followed by uniform sampling among $E^\mu(\lambda)$) give a proper weighted sampling among all vectors in $E(\lambda)$ with $v \in E^{\mu}(\lambda)$ having probability proportional to $Q_{(1^n)}^{\mu}(q)$.

With the $v \in E^\mu(\lambda)$ chosen above, one then notes that the transformation $\overline{J}$ on $\mathbb{F}_q^n/\mathrm{span}(v)$ has Jordan type $\mu$. Inductively applying this algorithm, one obtains an $n$-step flag fixed by $J$ which we call $V_\bullet'$, writing
\begin{align*}
    0 \subset V_1' = \mathrm{span}(v) \subset V_2' \subset \dots \subset V_n' = \mathbb{F}_q^n.
\end{align*}
By Proposition \ref{prop:pointcount} applied to $\overline{J}$, there are $Q_{(1^n)}^\mu(q)$ such flags fixed by $\overline{J}$ whose first step $V_1$ is $\mathrm{span}(v)$. For the sake of induction, we assume for now that each of these flags is chosen uniformly at random.

Since our algorithm chooses $v$ with probability proportional to $Q_{(1^n)}^\mu(q)$, which is the number of possible extensions of $\mathrm{span}(v)$ into an $n$-step flag which are fixed by $J$, we then get that our algorithm chooses uniformly at random among flags fixed by $J$. The base case where $n = 1$ is clearly a uniform sampling.

\subsubsection{Conclusion}

We then know that the $n$-step flag
\begin{align}
    0 \subset QV_1' \subset \dots \subset QV_n' = \mathbb{F}_q^n
\end{align}
has been chosen uniformly at random among flags fixed by $a = QJQ^{-1}$. Let $g'B$ be the corresponding point in $G/B$, with $g'$ constructed as in Section \ref{sec:bruhatgaussian}. By the method described therein, we can also then write $g' = u'w'B$ for some $u' \in U$ and $w' \in S_n$.

Starting with the original point $gB$, we have just explained how to sample an element $a$ uniformly from its $U$-stabilizer, and then an element of $G/B$ uniformly among $a$-fixed points. Therefore this algorithm indeed samples according to the Burnside process for the $U$-action on $G/B$.

Writing $gB = uwB$ and $g'B = u'w'B$ and recording only the permutations $w$ and $w'$, this therefore also describes a sampling algorithm for the lumped Burnside process orbits, i.e.\ the corresponding Markov chain on $S_n$.

\section{\texorpdfstring{$\mathrm{GL}_n$}{GLn} behavior: Tableaux and the Robinson--Schensted correspondence}\label{sec:typea}

In this section, we prove our main result for $\mathrm{GL}_n(\mathbb{F}_q)$, in the form of Theorem \ref{thm:limit}, describing the large-$q$ limiting behavior of the Markov chain on $S_n$ arising from the Burnside process on the flag variety in terms of the combinatorics of the Robinson--Schensted correspondence.

Before doing so, we now describe this behavior informally. We will explain that to any permutation $S_n$, one can associate a standard Young tableau $P(w)$ with $n$ boxes. We will show that when $q$ is large, then with very high probability, the Markov chain sends any $w \in S_n$ to a uniform choice of $z \in S_n$ subject to the condition that $P(z) = P(w)$. Thus the Markov chain tends to cluster into ``buckets" labelled by Young tableaux which are very difficult to exit. Despite this, over a very long period of time, the Markov chain still eventually finds its way from bucket to bucket so as to converge to the uniform distribution.

In Section \ref{sec:generaltype}, we will show that the same behavior holds in the more general setting when $\mathrm{GL}_n(\mathbb{F}_q)$ is replaced by a finite Chevalley group $G(\mathbb{F}_q)$ of any type, thereby obtaining a Markov chain on an arbitrary Weyl group $W$. We will then show that the same behavior roughly holds, but the ``buckets" from the previous paragraph can no longer be labeled by Young tableaux. Instead, they will be indexed by \emph{Steinberg cells}, a certain partition of $W$ which arises from the geometry of orbital varieties.

\subsection{The Robinson--Schensted correspondence}

The Robinson--Schensted correspondence is a bijection between permutations $w \in S_n$ and pairs of standard Young tableaux $(P, Q)$ of the same shape $\lambda \vdash n$. Given a permutation $w$, the correspondence produces two tableaux:
\begin{itemize}
    \item $P(w)$, called the \emph{insertion tableau}, constructed by successively inserting the entries of $w$ into a tableau by the row-insertion algorithm described in \cite{Schensted}.
    \item $Q(w)$, called the \emph{recording tableau}, which records the order in which boxes are added during the insertion process.
\end{itemize}
We refer the reader to \cite{Fulton} for a clear summary and exposition of the algorithm in question.  

Both $P(w)$ and $Q(w)$ are standard Young tableaux of the same shape $\lambda$. We let $T(w) = \lambda$ be the corresponding partition of $n$. We now recall the following basic property of the Robinson--Schensted correspondence.
\begin{proposition}[Lemma 7 in \cite{Schensted}]
    For any $w \in S_n$,
    \begin{align}
        P(w^{-1}) & = Q(w).
    \end{align}
\end{proposition}

\begin{definition}
For a partition $\lambda \vdash n$, recall that $f^\lambda$ is the number of standard Young tableaux of shape $\lambda$. Equivalently, $f^\lambda$ is the dimension of the irreducible representation of the symmetric group $S_n$ corresponding to $\lambda$. The hook length formula gives
\[
f^\lambda \;=\; \frac{n!}{\prod_{(i,j) \in \lambda} h_{ij}},
\]
where $h_{ij}$ is the hook length of the box $(i,j)$ in the Young diagram of $\lambda$.
\end{definition}

\subsection{Underlying geometry}

Recall that in Proposition \ref{prop:glnuw}, we described the subspace $U_w = U \cap wUw^{-1}$ of matrices fixing $wB \in G/B$. It will be convenient to instead work with $\mathfrak{n}$, the set of strictly upper-triangular (nilpotent) $n\times n$ matrices. We let $\mathfrak{n}^w = w\mathfrak{n}w^{-1}$ and write $\mathfrak{n}_w = \mathfrak{n} \cap \mathfrak{n}^w$. We note that
\[\mathfrak{n}_w = \{u - \mathrm{id}_{n\times n} : u \in U_w\}.\]
Since a flag $V_\bullet$ is fixed by $u \in U_w$ if and only if it is fixed by $x = u - \mathrm{id}_{n\times n} \in \mathfrak{n}_w$, we note that $U_w$ and $\mathfrak{n}_w$ can and will be used interchangeably for the purposes of the Burnside process. Note that $\mathfrak{n}$ and its subspaces admit an action by $G$ under conjugation. We note that $\mathfrak{n}$, its subspaces, and $G$-orbits therein all naturally inherit the structures of varieties over $\overline{\mathbb{F}}_q$, which we will use in what follows to discuss notions of genericity and dimension.

\begin{definition}
    For any $w \in W$, let
    \begin{align*}
        V_w = U\cdot \mathfrak{n}_w.
    \end{align*}
\end{definition}
Note that when considered as an algebraic variety over $\overline{\mathbb{F}}_q$, the closure $\overline{V_w}$ is an \emph{orbital variety}: an irreducible component of the intersection of a nilpotent orbit with $\mathfrak{n}$. The following is a restatement of Steinberg's beautiful geometric realization of the Robinson--Schensted correspondence which appears in \cite{Steinberg}.
\begin{prop}[Steinberg, \cite{Steinberg}]\label{prop:pstein}
    For a pair of permutations $w, z \in S_n$, $P(w) = P(z)$ if and only if
    \begin{align}\label{eqn:unwuny}
        U \cdot \mathfrak{n}_w = U \cdot \mathfrak{n}_z.
    \end{align}
\end{prop}
An alternative way to see this result is as follows. In the forthcoming Definition \ref{def:steincell}, we recall the definition of \emph{right Steinberg cells} in general Weyl groups and explain that by definition, $w$ and $y$ are in the same right Steinberg cell if and only if (\ref{eqn:unwuny}) holds. Proposition \ref{prop:pstein} then asserts that for $S_n$, right Steinberg cells are determined by the $P$-symbol of a permutation. Indeed, in \cite{BorBry}, it is shown that for $S_n$, right Steinberg cells agree with right Kazhdan--Lusztig cells; our result then follows from Kazhdan--Lusztig's original work \cite{KL} where right Kazhdan--Lusztig cells in $S_n$ are indexed by standard Young tableaux, and membership of an element $w$ in a cell indexed by $P$ is decided by the condition that $P(w) = P$.

\begin{lemma}\label{lem:genericjordan}
    When $q$ is large, a generic element of $\mathfrak{n}_w$ (considered as a variety over $\mathbb{F}_q$) has Jordan type given by $T(w)$, the shape of $P(w)$. 
    
    Further, any arbitrary element of $\mathfrak{n}\cap \mathfrak{n}^w$ has Jordan type $\mu$ for some $\mu \leq T(w)$ in the dominance order for partitions.
\end{lemma}
\begin{proof}
    By definition, a generic element of $\mathfrak{n} \cap \mathfrak{n}^w$ has a zero in entry $(i, j)$ if $(i,j)$ is not in the inversion set of $w$ and a generic nonzero element of $\mathbb{F}_q$ if $(i,j)$ is in the inversion set of $w$ for $i < j$. When $q$ is large, one can choose generic nonzero elements such that the Jordan type of the corresponding matrix has the same Jordan type as a generic element (over $\mathbb{C}$) of the subspace of $\mathfrak{n}$ corresponding to the \emph{inversion digraph of $w$} as in \cite{Gan}.

    Gansner's theorem, explained in loc.\ cit., gives a partition from any generic nilpotent matrix with support determined by a digraph by taking the Jordan type. The inversion digraph is precisely the complement of the poset of allowed entries for $\mathfrak{n} \cap \mathfrak{n}^{w}$. Gansner explains that when this procedure is applied to a generic nilpotent matrix whose support is determined by the the inversion digraph of a permutation, one obtains the transpose of the underlying Young diagram obtained from applying Robinson--Schensted algorithm, as detailed in \cite{Schensted}, to that same permutation. This relationship is established using Greene's theorem from \cite{Greene}. Thus applying it to the complement (the non-inversion digraph), we obtain simply $T(w)$ without a transpose; this is exactly the first claim in the lemma.

     Now note that the same logic combined with \cite[Lemma 4.1]{Gan} implies that any matrix (generic or not) in $\mathfrak{n} \cap \mathfrak{n}^w$ must have Jordan type dominated by $T(w)$ in the dominance order. 
\end{proof}

\begin{definition}
    For any $w \in S_n$, let $d(w) = \dim V_w$ as a variety over $\overline{\mathbb{F}}_q$. 
\end{definition}

The following explicit computation for $d(w)$ follows from \cite{BorBry}.
\begin{prop}\label{prop:dworbit}
    There exists some generic element $x \in \mathfrak{n}_w$ such that
    \begin{align}
        d(w) & = \frac{1}{2}\dim \mathbb{O}_x,
    \end{align}
    where $\mathbb{O}_x$ is the nilpotent orbit $G\cdot x$.
\end{prop}

We note that the nilpotent orbit $\mathbb{O}_x$ depends only on the Jordan type $\lambda(x)$, so it makes sense to simply refer to a nilpotent orbit $\mathbb{O}_\lambda$ for $\lambda$ a partition of $n$. Applying Lemma \ref{lem:genericjordan} to the dimension formula in Proposition \ref{prop:dworbit}, we get the following.
\begin{corollary}\label{cor:dwnilp}
    For any $w \in S_n$, let $\lambda = T(w)$. Then $d(w)$ depends only on $\lambda$, and we have
    \begin{align}
        d(w) & = \frac{\dim \mathbb{O}_\lambda}{2}
    \end{align}
\end{corollary}

The dimension $\dim \mathbb{O}_\lambda$ can be described combinatorially as follows (see, e.g., \cite[III.7]{CMcG}). 
\begin{lemma}\label{prop:nilpdim}
    For any partition $\lambda$ of $n$, the dimension of the nilpotent orbit $\mathbb{O}_\lambda$ is
    \begin{align}
        \dim \mathbb{O}_\lambda & = n^2 - \sum_i (\lambda_i')^2,
    \end{align}
    where $\lambda'$ is the transpose (conjugate partition) of $\lambda$.
\end{lemma}

\subsection{Combinatorics of Green polynomials}
Recall that for any $\lambda \vdash n$ with conjugate partition $\lambda'$,
    \begin{align}
        b(\lambda) & = \deg Q_{(1^n)}^\lambda(q) = \sum_{i \geq 1} (i - 1)\lambda_i = \sum_j \binom{\lambda_j'}{2}.\label{eqn:macdonaldblambda}
\end{align}

Comparing this with Lemma \ref{prop:nilpdim} gives us the following comparison between $b(\lambda)$ and $d(w)$, whenever $P(w)$ has shape $\lambda$.

\begin{lemma}\label{lem:blambda}
    For any partition $\lambda$ of $n$ and any $w \in S_n$ such that $P(w)$ has shape $\lambda$, we have
    \begin{align}
        b(\lambda) & = \binom{n}{2} - d(w).
    \end{align}
\end{lemma}
\begin{proof}
    Let $\lambda \vdash n$ with conjugate partition $\lambda'$. Consider the nilpotent orbit $\mathbb{O}_\lambda \subset \mathfrak{gl}_n$ of Jordan type $\lambda$. By Proposition \ref{prop:nilpdim}, its dimension is given by
    \begin{align}
        \dim \mathbb{O}_\lambda = n^2 - \sum_i (\lambda'_i)^2.
    \end{align} 

    Let $\mathbb{O}_\text{reg}$ denote the regular nilpotent orbit (corresponding to the Jordan type $(1, 1, \dots, 1)$), which has dimension
    \[
        \dim \mathbb{O}_\text{reg} = n^2 - n.
    \]
    Then we can rewrite
    \begin{align}
        \sum_i (\lambda'_i)^2 = n + 2 \sum_j \binom{\lambda'_j}{2} = n + 2 b(\lambda),
    \end{align}
    so that
    \begin{align}
        \dim \mathbb{O}_\lambda = n^2 - \sum_i (\lambda'_i)^2 = n^2 - n - 2 b(\lambda) = \dim \mathbb{O}_\text{reg} - 2 b(\lambda).
    \end{align}

    By Corollary \ref{cor:dwnilp}, $d(w) = \frac{1}{2}\dim \mathbb{O}_\lambda$. Combining the above, we get
    \begin{align}
        b(\lambda) = \frac{\dim \mathbb{O}_{\mathrm{reg}}}{2} - d(w).
    \end{align}
    Noting that $\binom{n}{2} = \frac{n^2 - n}{2} = \frac{\dim \mathbb{O}_\mathrm{reg}}{2}$, we obtain
    \begin{align}
        b(\lambda) = \ell(w_0) - d(w).
    \end{align}
\end{proof}

By the formula (\ref{eqn:macdonaldblambda}), one can directly check the following.
\begin{lemma}\label{lem:blambdadominance}
    For any $\mu \leq \lambda$ in the dominance order we have $b(\mu) \leq b(\lambda)$
\end{lemma}

\subsection{The \texorpdfstring{$q \to \infty$}{q to infinity} limit}
We arrive at our main theorem, which is the following combinatorial explanation of the behavior of the Markov chain on $S_n$ for large $q$.
\begin{theorem}\label{thm:limit}
    For $w, z \in S_n$, writing $\lambda$ for the shape of $P(w)$, we have
    \begin{align}
        \lim_{q \to \infty} P_{S_n}(w, z) & = \begin{cases}
            \frac{1}{f^\lambda} & P(w) = P(z),\\
            0 & P(w) \neq P(z).
        \end{cases}
    \end{align}
    In fact,\[P_{S_n}(w, z) = \frac{1}{f^\lambda}\delta_{P(w), P(z)} + O(1/q).\]
\end{theorem}

In this section, we develop some intermediate results which will be required for the proof of Theorem \ref{thm:limit}.
\begin{lemma}\label{lem:stabilizer}
    For any $z \in S_n$, the cardinality of the stabilizer of $\mathfrak{n}_z$ under the adjoint action of $U$ on the Grassmannian of $(\binom{n}{2} - \ell(z))$-dimensional linear subspaces of $U \cdot \mathfrak{n}^z$ is
    \begin{align}
        \frac{|U||\mathfrak{n}_z|}{|V_z|}
    \end{align}
\end{lemma}
\begin{proof}
    Consider the action of $U$ by conjugation on subspaces of $V_z$. We note that $|V_z| = |\mathfrak{n}_z||\mathbb{O}|$, where $\mathbb{O}$ is the orbit of $U$ on $\mathfrak{n} \cap \mathfrak{n}^z$ as a subset of the Grassmannian. This means $|\mathbb{O}| = |V_z|/|\mathfrak{n}_z|$. On the other hand, by the orbit-stabilizer theorem, the stabilizer has cardinality $|U|/|\mathbb{O}|$, giving the result.
\end{proof}

\begin{proposition}\label{prop:mainapprox}
For any $w \in S_n$, we have
    \begin{align}
        \lim_{q \to \infty}  \frac{1}{|U_w|} \sum_{a \in U_w} \frac{q^{b(\lambda)}}{|X_a|} = \frac{1}{f^\lambda},
    \end{align}
    where $\lambda$ is the shape of $P(w)$.
\end{proposition}
\begin{proof}
    Let $w \in S_n$ and let $\lambda$ be the shape of $P(w)$. We note that the statement remains identical if $U_w$ is replaced by $\mathfrak{n}_w = U_w - \mathrm{id}_{n\times n}$. 
    Let $\mathfrak{n}_w^{\circ}$ be the generic subset of $\mathfrak{n}_w$ consisting of elements with Jordan type $\lambda$.

     First note that for generic $a \in V_w$, we have by Proposition \ref{prop:pointcount} and Lemma \ref{lem:genericjordan} that $|X_a| = Q_{(1^n)}^\lambda(q)$, which has degree $b(\lambda)$ by Lemma \ref{lem:blambda}. Further, by Lemmas \ref{lem:genericjordan} and \ref{lem:blambdadominance}, for any $a' \in V_w$ not generic in this sense, we have that $|X_a|$ is a polynomial in $q$ with degree at least $b(\lambda)$. This means that, up to some $O(1/q)$ terms, we have
    \begin{align}
        \frac{1}{|\mathfrak{n}_w|} \sum_{a \in U_w} \frac{q^{b(\lambda)}}{|X_a|} & \approx \frac{1}{|\mathfrak{n}_w|} \sum_{a \in \mathfrak{n}_w^\circ} \frac{q^{b(\lambda)}}{|X_a|}\\
        & = \frac{|\mathfrak{n}_w^\circ|}{|\mathfrak{n}_w|}\cdot \frac{q^{b(\lambda)}}{Q_{(1^n)}^\lambda(q)}\\
        & \approx \frac{q^{b(\lambda)}}{Q_{(1^n)}^\lambda(q)}
        \label{eqn:oneoverq}
    \end{align}
    with the last step holding since the number of elements of $U_w^\circ$, an open subset of the irreducible variety $V_w$, is $q^{d(w)}$ up to a linear combination of smaller powers of $q$. Recalling as in (\ref{eqn:recipxq}) that 
    \begin{align*}
        q^{-b(\lambda)}Q_{(1^n)}^\lambda(q) & = X_{(1^n)}^\lambda(1/q),
    \end{align*}
    by the Taylor series expansion of (\ref{eqn:oneoverq}) at $q = \infty$, we see that its leading (degree zero) term is exactly $(X_{(1^n)}^\lambda(0))^{-1}$. Since by Proposition \ref{prop:flambda}, $X_{(1^n)}^\lambda(0) = f^\lambda$, putting this together gives that
    \begin{align*}
        \lim_{q \to \infty} \frac{1}{q^{d(w) - b(\lambda)}}\sum_{a \in V_w} \frac{1}{|X_a|} & = \lim_{q \to \infty} \frac{1}{q^{-b(\lambda)}Q_{(1^n)}^\lambda(q)},\\
        & = \frac{1}{X_{(1^n)}^\lambda(0)} \\
        & = \frac{1}{f^\lambda},
    \end{align*}
    as claimed in the proposition.
\end{proof}

We now use this estimate to prove Theorem \ref{thm:limit}.

\begin{proof}[Proof of Theorem \ref{thm:limit}]
    Suppose that $w, z \in W$ are such that $V_w = V_z$. Then by Lemma \ref{lem:stabilizer}, Lemma \ref{lem:blambda}, and Proposition \ref{prop:mainapprox} applied in succession, we have
    \allowdisplaybreaks[1]
    \begin{align*}
        P_{S_n}(w, z) & = \frac{1}{|U_z|}\sum_{u \in U} P(w, uz)\\
        & = \frac{1}{|U_z||U_w|} \sum_{u \in U} \sum_{a \in U_w \cap uU_zu^{-1}} \frac{1}{|X_a|}\\
        & = \frac{1}{|\mathfrak{n}_z||\mathfrak{n}_w|} \sum_{u \in U} \sum_{a \in \mathfrak{n}_w \cap u \cdot \mathfrak{n}_z} \frac{1}{|X_a|}\\
        & \approx \frac{1}{|\mathfrak{n}_z||\mathfrak{n}_w|} \cdot \frac{|U||\mathfrak{n}_z|}{|\mathfrak{n}_z|} \sum_{a \in \mathfrak{n}_w \cap V_z} \frac{1}{|X_a|}\\
        & = \frac{1}{|\mathfrak{n}_w|} \cdot \frac{|U|}{|V_z|} \sum_{a \in \mathfrak{n}_w \cap V_w} \frac{1}{|X_a|}\\
        & \approx \frac{1}{|\mathfrak{n}_w|} \sum_{a \in \mathfrak{n}_w} \frac{q^{\ell(w_0) - d(w)}}{|X_a|} \\
        & = \frac{1}{|\mathfrak{n}_w|} \sum_{a \in U_w} \frac{q^{b(\lambda)}}{|X_a|}\\
        & \approx \frac{1}{f^\lambda},\\
    \end{align*}
up to terms of order $O(1/q)$, for $\lambda$ the shape of $P(w)$. A similar direct computation shows that
\begin{align}
    P_{S_n}(w, z) = O(1/q)
\end{align}
whenever $P(w) \neq P(z)$. Alternatively, we note that for any fixed $w \in S_n$, the above implies that 
\begin{align}
    \lim_{q \to \infty} \sum_{\substack{z\in S_n\\ P(w) = P(z)}} P_{S_n}(w, z) & =  f^\lambda \cdot \frac{1}{f^\lambda} = 1,
\end{align}
and so we must have $P_{S_n}(w, z) = 0$ for all $z$ for which $P(w) \neq P(z)$.
\end{proof}

\section{General type behavior: orbital varieties and Steinberg cells}\label{sec:generaltype}

Although our result for $\mathrm{GL}_n(\mathbb{F}_q)$ can be stated in terms of the combinatorics of the Robinson--Schensted correspondence, the behavior we will describe in Theorem \ref{thm:gentypelimit} is a phenomenon which occurs in general type, for $G$ a finite Chevalley group. To state and prove it, we introduce the notion of Steinberg cells, which will serve as analogues of the subsets of $S_n$ determined by a fixed choice of $P(w)$ or $Q(w)$, but for general-type Weyl groups $W$. Once we set up the geometric analogues of the combinatorial objects studied in Section \ref{sec:typea} in general type, we can immediately use the same analysis as in the proof of Theorem \ref{thm:limit} to establish its general-type analogue in the form of Theorem \ref{thm:gentypelimit}.

\subsection{Setup for Chevalley groups of general type}

One of the great achievements of twentieth-century mathematics is
the unified treatment of a vast class of finite simple groups as ``finite groups of Lie type." Of course $\mathrm{GL}_n(\mathbb{F}_q)$ is among them (it has a simple quotient) as are the other classical orthogonal, unitary and symplectic groups over $\mathbb{F}_q$. These along with with twisted variants and the exceptional groups $G_2,F_4, E_6, E_7, E_8$ have appearances and applications in all corners of mathematics. A readable introduction to all that is needed for the present paper is in \cite[\S 1]{Carter}. We now recall how to construct such a group over $\mathbb{F}_q$ in any Dynkin type, and we will see that our results concerning all of these groups can be treated at once with unified arguments.

Let $\mathfrak{g}$ be a semisimple Lie algebra over $\mathbb{C}$. A construction of Chevalley \cite{Chevalley} ensures that one can construct an affine group scheme $\mathbf{G}$ over $\mathbb{Z}$ corresponding to $\mathfrak{g}$. We fix a Cartan subgroup $\mathbf{T}$, a Borel subgroup $\mathbf{B} \subset \mathbf{G}$ containing $\mathbf{T}$, and we let $\mathbf{U} \subset \mathbf{G}$ denote its unipotent radical. These have corresponding Lie subalgebras $\mathfrak{h}, \mathfrak{b}, \mathfrak{n}$ of $\mathfrak{g}$.

Associated to $\mathfrak{g}$ is a root system $\Phi$; we write $\Phi^+$ for its positive roots. Let $W = N_\mathbf{G}(\mathbf{T})/\mathbf{T}$ be the associated Weyl group, and let $\ell : W \to \mathbb{Z}_{\geq 0}$ be the length function. Let $w_0 \in W$ be the longest element. For any $\alpha \in \Phi^+$, there is a corresponding $1$-parameter subgroup $\mathbf{U}_\alpha$ of $\mathbf{U}$; similarly let $\mathfrak{g}_\alpha$ denote the corresponding root subspace of $\mathfrak{g}$.

Fix once and for all a prime power $q$, and let $G = \mathbf{G}(\mathbb{F}_q)$, $B = \mathbf{B}(\mathbb{F}_q)$, $U = \mathbf{U}(\mathbb{F}_q)$, and $U_\alpha = \mathbf{U}_\alpha(\mathbb{F}_q)$ for any $\alpha \in \Phi^+$.

\subsubsection{The Burnside process on the flag variety}

Letting \[X = \mathbf{G}(\mathbb{F}_q)/\mathbf{B}(\mathbb{F}_q) = G/B,\] we continue to call this the \emph{flag variety} of $\mathbf{G}$ over the finite field $\mathbb{F}_q$. In this setting we again have the Bruhat decomposition.
\begin{prop}[see \cite{Carter}, \S 1.10]\label{prop:bruhatlater}
    There is a decomposition
    \begin{align*}
        X & = \coprod_{w \in W} BwB = \coprod_{w \in W} UwB,
    \end{align*}
    where for each $w \in W$, $UwB \cong \mathbb{A}^{\ell(w)}_{\mathbb{F}_q}$ is thought of as a subset of $X$ and so $|UwB| = q^{\ell(w)}$.
\end{prop}

We now study the Burnside process for the action of $U$ on $X$ in this more general setup. We again will write $P(x, y)$ for the transition probabilities, and since by Proposition \ref{prop:bruhat} the orbits for this action are indexed by $W$, for $w, z \in W$ we will use $P_W(w, z)$ to denote the transition probabilities for the corresponding Burnside process on the Weyl group.

\subsection{Stabilizers in \texorpdfstring{$U$}{U}} We now give a more general description of $\mathrm{stab}_U(wB)$ for any $w \in W$, generalizing Proposition \ref{prop:glnuw}. Given $w \in W$, we again set $U^w = wUw^{-1} \subset G$, $\mathfrak{n}^w = w\mathfrak{n}w^{-1} \subset \mathfrak{g}$, and $U_w = \mathrm{stab}_U(wB) = U \cap U^w$.

\begin{prop}[{see \cite[\S 2.5]{Carter}}]\label{prop:uw}
    For any $w \in W$, the stabilizer $U_w$ is the subgroup of $U$ generated by those root subgroups $U_\alpha$ with $\alpha \in \Phi^+$ such that $w^{-1}\alpha \in \Phi^+$. Equivalently,
    \[
        U_w \;=\; \prod_{\substack{\alpha \in \Phi^+ \\ w^{-1}\alpha \in \Phi^+}} U_\alpha.
    \]
    Similarly,
    \[
        \mathfrak{n}_w = \mathfrak{n} \cap \mathfrak{n}^w
        = \bigoplus_{\substack{\alpha \in \Phi^+ \\ w^{-1}\alpha \in \Phi^+}} \mathfrak{g}_\alpha.
    \]
\end{prop}
In particular, since the cardinality of the indexing set in the product and direct sum which appear in Proposition \ref{prop:uw} is exactly $\ell(w_0) - \ell(w) = \ell(w_0w)$, we will use that for any $w \in W$,
\begin{equation}\label{eqn:uwlater}
    |U_w| = q^{\ell(w_0) - \ell(w)}.
\end{equation}

\subsection{General type behavior}

As in Section \ref{sec:typea}, for any $w \in W$ let $V_w = U\cdot \mathfrak{n}_w$.

\begin{definition}\label{def:steincell}
    For $w \in W$, the \emph{right Steinberg cell} $\mathcal{S}(w)$ is the set of all $z \in W$ such that
    \[
       V_z = V_w
    \]
    That is, $\mathcal{S}(w)$ consists of all $z \in W$ for which the $U$-orbit of $\mathfrak{n} \cap \mathfrak{n}^z$ for the adjoint action coincides with that of $\mathfrak{n} \cap \mathfrak{n}^w$.
\end{definition}

\begin{definition}
In \cite{BorBry}, it is explained that $\mathfrak{n}_w$ has a dense open subset consisting of elements $x$ such that $\overline{V_w}$, considered now as a variety over $\overline{\mathbb{F}}_q$ is an irreducible component of $\overline{\mathbb{O}_x \cap \mathfrak{\mathfrak{n}}}$, where $\mathbb{O}_x$ is the nilpotent orbit $G\cdot x$. Let $\mathfrak{n}_w^\circ$ be this set.
\end{definition}

\begin{definition}
    Suppose $u \in \mathfrak{n}_w^\circ$. Let $b(w) = \dim X_u$, and let $f^w$ be the number of irreducible components of $X_u$.
\end{definition}

The following result, expressing a relationship between Springer fibers for $u \in U_w$ and the geometry of the varieties $V_w = U \cdot \mathfrak{n}\cap \mathfrak{n}^w$, is explained in \cite[Appendix B]{BorBry}.
\begin{prop}\label{prop:geometric}
 For $u \in \mathfrak{n}_w^\circ$, the number of irreducible components of $X_u$ is equal to $|\mathcal{S}(w)|$. Further, for any $u \in \mathfrak{n}_w$,
    \begin{align*}
        \dim X_u \leq b(w).
    \end{align*} 
\end{prop}

\begin{example}
    In Type $A$, we note that $b(w) = b(\lambda)$. Further, our comments after Proposition \ref{prop:pointcount} guarantee that in this case, the irreducible components of $X_u$ for $u \in \mathfrak{n}_w^\circ$ is equal to $|\mathcal{S}(w)| = f^\lambda$ for $\lambda$ the shape of $P(\lambda)$
\end{example}

The following is then a direct consequence of Proposition \ref{prop:geometric}.
\begin{corollary}
    For any $w \in W$,
    \begin{align}
        \lim_{q \to \infty} \frac{1}{|U_w|}\sum_{u \in U_w} \frac{q^{b(w)}}{|X_u|} & = \frac{1}{f^{w}}
    \end{align}
\end{corollary}

The very same proof we used to establish Theorem \ref{thm:limit} can then be used to show the following generalization.
\begin{theorem}\label{thm:gentypelimit}
    For any $w, z \in W$,
    \begin{align*}
        \lim_{q \to \infty} P_W(w, z) & = \begin{cases}
        \frac{1}{f^{\mathcal{S}(w)}} & z \in \mathcal{S}(w)\\
        0 & z \not\in \mathcal{S}(w)
    \end{cases}
    \end{align*}
    In fact,\[P_{W}(w, z) = \frac{1}{f^\lambda}\delta_{\mathcal{S}(w), \mathcal{S}(z)} + O(1/q).\]
\end{theorem}

\section{Consequences for mixing times}\label{sec:mixing}
We now explain that the leading-order behavior of the Burnside process described in Sections \ref{sec:typea} and \ref{sec:generaltype} gives a lower bound on the mixing time.

For a given Weyl group $W$, let $\mathcal{S}_1, \dots, \mathcal{S}_k$ be an enumeration of the right Steinberg cells. By the analysis done in Section \ref{sec:generaltype}, there exists some constant $C'$ not depending on $q$ such that for any $w, z$ with $\mathcal{S}(w) \neq \mathcal{S}(z)$,
\begin{align}
    P_W(w, z) & \leq \frac{C'}{q}.
\end{align}
In other words, there is a constant $C$ not depending on $q$ such that for any $w \in W$,
\begin{align}
    P_W(\mathcal{S}(w), W - \mathcal{S}(w)) & \leq \frac{C}{q}.
\end{align}
We will use this to prove the following lower bound on the mixing time of the Markov chain.
\begin{theorem}\label{thm:mixing}
    For any $\varepsilon > 0$, the $\varepsilon$-mixing time $\tau_{\mathrm{mix}}(\varepsilon)$ of the Burnside process on $W$ satisfies
    \begin{align*}
        \tau_{\mathrm{mix}}(\varepsilon)& \geq \frac{q}{2C}\log\left(\frac{1}{2\varepsilon}\right) = \Omega(q\log(1/\varepsilon)).
    \end{align*}
\end{theorem}
\begin{proof}
    Pick any Steinberg cell $S$ and note that $\pi(S)\le 1/2$; since the stationary distribution of the chain is uniform, we have $\pi(x)=1/|W|$ for all $w\in W$ and hence $\pi(S)=|S|/|W|$. For any $w\in S$, we know the probability $P(w, S^c)$ of leaving $S$ in one step is bounded by $C/q$. Now consider the stationary flow out of $S$:
\[
Q(S,S^c)=\sum_{w\in S}\pi(w)P(w,S^c)
\leq \frac{C}{q}\pi(S).
\]
Therefore the conductance of $S$ satisfies
\[
\Phi(S)= \frac{Q(S,S^c)}{\pi(S)} \leq \frac{C}{q}.
\]
Since $\Phi = \min_{T:0<\pi(T)\le 1/2}\Phi(T)$, we conclude $\Phi \le C/q$.

By the standard conductance lower bound (see, e.g., \cite[Theorem~7.4]{MarkovMixing}),
\[
\tau_{\mathrm{mix}}(\varepsilon)\;\ge\;\frac{1}{2\Phi}\ln\frac{1}{2\varepsilon}.
\]
Plugging in $\Phi \le C/q$ yields
\[
\tau_{\mathrm{mix}}(\varepsilon)\geq \frac{q}{2C}\ln\frac{1}{2\varepsilon}
=\frac{q}{2C(|W|-1)}\ln\frac{1}{2\varepsilon}.
\]

In particular, $\tau_{\mathrm{mix}}(\varepsilon)=\Omega(q\log(1/\varepsilon))$ as claimed.
\end{proof}
Theorem \ref{thm:mixing} guarantees that, by choosing $q$ large, the Burnside process on $W$ defined in the present paper provides a natural family of Markov chains on the Weyl group which can be made to mix ``arbitrarily slowly."

\section{Examples and simulations}\label{sec:sim}

\subsection{Explicit transition probabilities for \texorpdfstring{$\mathrm{GL}_3(\mathbb{F}_q)$}{GL3(Fq)}}
    Using a computer, we produced the transition matrix for $G = \mathrm{GL}_3(\mathbb{F}_q)$ for arbitrary $q$, expressing each transition probability as a rational function in $q$. In this case, $W \cong S_3$, and we again label its elements in the order
    \[\{e, s_1, s_2, s_2s_1, s_1s_2, w_0\}.\]
    The transition matrix is as follows:
    \[\renewcommand{\arraystretch}{2}
\scalebox{0.70}{%
$\displaystyle
\begin{pmatrix}
\frac{q^3}{q^3 + 2q^2 + 2q + 1} & \frac{q^3 + 2q^2 + q - 1}{2q^4 + 5q^3 + 6q^2 + 4q + 1} & \frac{q^3 + 2q^2 + q - 1}{2q^4 + 5q^3 + 6q^2 + 4q + 1} & \frac{q^3 + q^2 + 1}{2q^4 + 5q^3 + 6q^2 + 4q + 1} & \frac{q^3 + q^2 + 1}{2q^4 + 5q^3 + 6q^2 + 4q + 1} & \frac{1}{q^3 + 2q^2 + 2q + 1}\\
\frac{q^3 + 2q^2 + q - 1}{2q^4 + 5q^3 + 6q^2 + 4q + 1} & \frac{q^4 + 2q^3 + q^2 - q}{2q^4 + 5q^3 + 6q^2 + 4q + 1} & \frac{q^3 + q^2 + 1}{2q^4 + 5q^3 + 6q^2 + 4q + 1} & \frac{q^4 + q^3 + q}{2q^4 + 5q^3 + 6q^2 + 4q + 1} & \frac{1}{q^3 + 2q^2 + 2q + 1} & \frac{q}{q^3 + 2q^2 + 2q + 1}\\
\frac{q^3 + 2q^2 + q - 1}{2q^4 + 5q^3 + 6q^2 + 4q + 1} & \frac{q^3 + q^2 + 1}{2q^4 + 5q^3 + 6q^2 + 4q + 1} & \frac{q^4 + 2q^3 + q^2 - q}{2q^4 + 5q^3 + 6q^2 + 4q + 1} & \frac{1}{q^3 + 2q^2 + 2q + 1} & \frac{q^4 + q^3 + q}{2q^4 + 5q^3 + 6q^2 + 4q + 1} & \frac{q}{q^3 + 2q^2 + 2q + 1}\\
 \frac{q^3 + q^2 + 1}{2q^4 + 5q^3 + 6q^2 + 4q + 1} & \frac{q^4 + q^3 + q}{2q^4 + 5q^3 + 6q^2 + 4q + 1} & \frac{1}{q^3 + 2q^2 + 2q + 1} & \frac{q^4 + q^3 + 2q^2 - 1}{2q^4 + 5q^3 + 6q^2 + 4q + 1} & \frac{q}{q^3 + 2q^2 + 2q + 1} & \frac{q^2}{q^3 + 2q^2 + 2q + 1}\\
\frac{q^3 + q^2 + 1}{2q^4 + 5q^3 + 6q^2 + 4q + 1} & \frac{1}{q^3 + 2q^2 + 2q + 1} & \frac{q^4 + q^3 + q}{2q^4 + 5q^3 + 6q^2 + 4q + 1} & \frac{q}{q^3 + 2q^2 + 2q + 1} & \frac{q^4 + q^3 + 2q^2 - 1}{2q^4 + 5q^3 + 6q^2 + 4q + 1} & \frac{q^2}{q^3 + 2q^2 + 2q + 1}\\
\frac{1}{q^3 + 2q^2 + 2q + 1} & \frac{q}{q^3 + 2q^2 + 2q + 1} & \frac{q}{q^3 + 2q^2 + 2q + 1} & \frac{q^2}{q^3 + 2q^2 + 2q + 1} & \frac{q^2}{q^3 + 2q^2 + 2q + 1} & \frac{q^3}{q^3 + 2q^2 + 2q + 1}
\end{pmatrix}
$}.
\]

Even in this small example, we see that the transition probabilities are difficult to understand algebraically, and exhibit some of the ``wild" behavior described in Section \ref{sec:wild}. However, we note that looking at their leading terms reveals the combinatorial patterns described in Section \ref{sec:typea}. Indeed, the $q \to \infty$ limit of this matrix is the matrix
\[
\begin{blockarray}{ccccccc}
      & e & s_1 & s_2 & s_2s_1 & s_1s_2 & w_0 \\[0.4em]
\begin{block}{c(cccccc)}
  e     & 1 & 0 & 0 & 0 & 0 & 0 \\
  s_1   & 0 & \tfrac{1}{2} & 0 & \tfrac{1}{2} & 0 & 0 \\
  s_2   & 0 & 0 & \tfrac{1}{2} & 0 & \tfrac{1}{2} & 0 \\
  s_2s_1& 0 & \tfrac{1}{2} & 0 & \tfrac{1}{2} & 0 & 0 \\
  s_1s_2& 0 & 0 & \tfrac{1}{2} & 0 & \tfrac{1}{2} & 0 \\
  w_0   & 0 & 0 & 0 & 0 & 0 & 1\\
\end{block}
\end{blockarray}
\]

Unlike the wild entries in the first matrix, this limiting matrix can be explained by Theorem \ref{thm:limit} and the fact that under the Robinson--Schensted correspondence,
\begin{align}
    P(e) &= \vcenter{\hbox{\begin{ytableau}
1 & 2 & 3
\end{ytableau}}}\\
P(s_1), P(s_2s_1) &= \vcenter{\hbox{\begin{ytableau}
1 & 3\\
2 
\end{ytableau}}}\\
P(s_2), P(s_1s_2) &= \vcenter{\hbox{\begin{ytableau}
1 & 2\\
3
\end{ytableau}}}\\
P(s_1s_2s_1) &= \vcenter{\hbox{\begin{ytableau}
1\\
2\\
3
\end{ytableau}}}.
\end{align}

The beautiful mathematical structure underlying the transition
matrix above suggests that representation theory might offer a way
diagonalizing these Markov chains. Indeed, after a lot of massaging,
this happened for the simplest case of the Burnside process \cite{DLR}. One hallmark of a ``nice diagonalization" is ``nice eigenvalues" (something like rational numbers). Alas, for the example above, when $q=2$, the eigenvalues are $0, 1, \frac{2}{5}$, and the three irrational roots of the cubic polynomial
\[525 x^3 - 315x^2 + 50x - 2.\]
These don't inspire hope.

That being said, inspection of the matrix for general $q$ when $n=3$ reveals that the second-largest eigenvalue is $\frac{2q-2}{2q + 1}$ with (right) eigenvector $(0,1,-1,1,-1,0)$. Indeed, once observed, it is easy to check with a computer that this is an eigenvalue-eigenvector pair. There
is an eigenvalue 0 with eigenvector $(0, 1, -1, -1, 1, 0)$, and of course $1$ is an eigenvalue with eigenvector $(1,1,1,1,1,1)$. This gives three roots of the characteristic polynomial. The remaining eigenvalues are roots of a cubic and standard analysis shows that for any $q$, each such root is smaller than $\frac{2q-2}{2q+1}$. All eigenvalues
are non-negative because of the structure of these two step chains. Using these observations, we can give sharp bounds on the rate of
convergence of the chain to stationarity on $S_3$ for general $q$. The argument shows that order $2q+1$ steps are necessary and sufficient for
convergence, from any start.

\begin{prop}\label{prop:gl3precise}
For the Burnside process on flags in $\mathrm{GL}_3(\mathbb{F}_q)$, the Markov chain on $S_3$ after $l$ steps, satisfies
\begin{align*}
    \frac{1}{2}\left(1 - \frac{3}{2q + 1}\right)^l \leq ||K^l(x, \cdot) - \pi(\cdot)|| \leq \sqrt{\frac{3}{2}}\left(1 - \frac{3}{2q+1}\right)^l
\end{align*}
Above, the total variation distance is
\begin{align*}
    \sum_{w \in S_3} \left|K^l(x, w) - \frac{1}{6}\right|.
\end{align*}
The upper  bound holds for all $x \in S_3$ and all $l \geq 1$. The lower bound holds for all $x$ except $\mathrm{id}$ and $w_0$ and all $l\geq 1$.
\end{prop}
\begin{proof}
The upper bound follows from the general upper bound (for any reversible chain on a finite state space with stationary distribution $\pi(y)$)
\begin{align*}
    4||K^l(x,\cdot) - \pi(\cdot)||^2 \leq \frac{1}{\pi_*}\beta_*^{2l}
\end{align*}
Here, $\pi_*= \min_{y \in W} \pi(y)$ and $\beta_*$ is the absolute value of the second eigenvalue. In the present case, $\pi_* = \frac{1}{6}$ and $\beta_* = \frac{2q-2}{2q+1}$. The claimed upper bound follows by elementary manipulation.

For the lower bound, use \cite[Lemma 2.1]{DKS} which states if $f(x)$ is an eigenvector with eigenvalue beta of a Markov chain with $\sup |f(x)|
= f_*$, then, for any $x$ and $l$,
\begin{align*}
    ||K^l(x,\cdot) - \pi(\cdot)|| \geq \frac{f(x)}{f_*} |\beta|^l.
\end{align*}
In the present case, the eigenfunction for $\frac{2q-2}{2q+1}$ takes values in $\{0, 1, -1\}$ with mean zero. 
\end{proof}
We remain curious about whether clean formulas for the second eigenvalue and its corresponding eigenvector may exist in other types.

\subsection{Simulations in larger types}

Using the algorithm described in Section \ref{sec:sampling}, we implemented the Burnside process in Type $A$ and ran some simulations for $\mathrm{GL}_4(\mathbb{F}_q)$ and $\mathrm{GL}_5(\mathbb{F}_q)$, some of which we include here to better illustrate the behavior we describe in Section \ref{sec:typea}. All of the figures in this sections are histograms depicting the number of times the Markov chain visited a given element in $S_n$. In these figures, we let $t$ be the number of steps for which the simulation was run. 

First, we note that when $q$ is small, we expect sampling from the Burnside process to be very close to sampling from the uniform distribution on $S_n$, as shown in Figures \ref{fig:first} and \ref{fig:second}. If $q$ is instead made to be large (compared to the number of steps), we then expect the uniform-like sampling property to disappear, instead being replaced by uniform sampling on the Steinberg cell, which in this case is the set of permutations with the same $P$-symbol as the initial element; this can be seen in Figures \ref{fig:third} and \ref{fig:fourth}.

Finally, Figure \ref{fig:fifth} shows an example where $q$ and $t$ have the same order of magnitude. Here we see that after remaining in a certain Steinberg cell for many steps, sometimes the chain exits and moves to a different Steinberg cell.

\begin{figure}[ht]
\caption{The histogram for $q = 3$, $n = 4$, $t = 1000$, with the chain starting at the permutation $3~ 2~1~ 4$.\label{fig:first}}
\begin{tikzpicture}
\begin{axis}[
    ymajorgrids,
    ymin=0, ymax=70,
    xmin=0.5, xmax=24.5,
    xtick=\empty,
    axis x line=bottom,
    axis y line=left,
    width=12cm,
    height=5cm,
    tick style={black},
    every axis plot/.append style={very thick}
]

\addplot+[ycomb, mark=*, mark options={fill=DarkPurple}, color=DarkPurple]
coordinates {
(1,36) (2,36) (3,47) (4,44) (5,37) (6,41) (7,47) (8,39)
(9,37) (10,41) (11,38) (12,43) (13,41) (14,41) (15,38) (16,36)
(17,47) (18,33) (19,56) (20,49) (21,46) (22,47) (23,45) (24,35)
};

\end{axis}
\end{tikzpicture}
\end{figure}
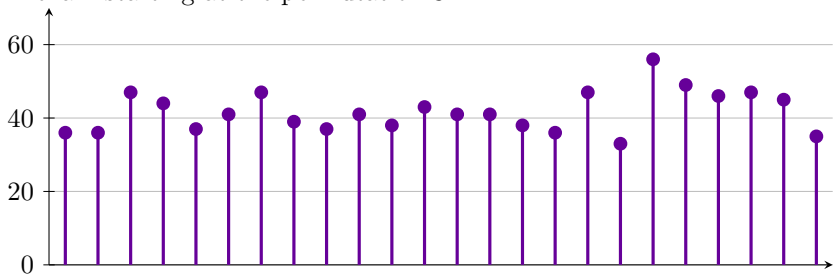

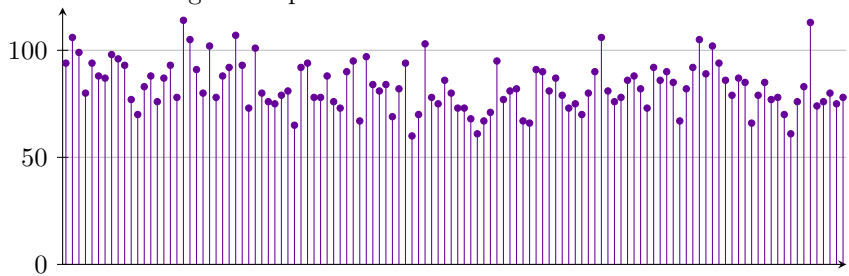
\begin{figure}[ht]
\caption{The histogram for $n = 5, q = 5, t = 1000$, with the chain starting at the permutation $3~ 2~1~ 4$.\label{fig:second}}
\begin{tikzpicture}
\begin{axis}[
    ymajorgrids,
    ymin=0, ymax=120,
    xmin=0.5, xmax=120.5,
    xtick=\empty,
    axis x line=bottom,
    axis y line=left,
    width=12cm,   
    height=5cm,
    tick style={black},
    every axis plot/.append style={thin},  
]

\addplot+[ycomb, mark=*, mark options={fill=DarkPurple, scale=0.6}, color=DarkPurple] 
coordinates {
(1,94) (2,106) (3,99) (4,80) (5,94) (6,88) (7,87) (8,98)
(9,96) (10,93) (11,77) (12,70) (13,83) (14,88) (15,76) (16,87)
(17,93) (18,78) (19,114) (20,105) (21,91) (22,80) (23,102) (24,78)
(25,88) (26,92) (27,107) (28,93) (29,73) (30,101) (31,80) (32,76)
(33,75) (34,79) (35,81) (36,65) (37,92) (38,94) (39,78) (40,78)
(41,88) (42,76) (43,73) (44,90) (45,95) (46,67) (47,97) (48,84)
(49,81) (50,84) (51,69) (52,82) (53,94) (54,60) (55,70) (56,103)
(57,78) (58,75) (59,86) (60,80) (61,73) (62,73) (63,68) (64,61)
(65,67) (66,71) (67,95) (68,77) (69,81) (70,82) (71,67) (72,66)
(73,91) (74,90) (75,81) (76,87) (77,79) (78,73) (79,75) (80,70)
(81,80) (82,90) (83,106) (84,81) (85,76) (86,78) (87,86) (88,88)
(89,82) (90,73) (91,92) (92,86) (93,90) (94,85) (95,67) (96,82)
(97,92) (98,105) (99,89) (100,102) (101,94) (102,86) (103,79) (104,87)
(105,85) (106,66) (107,79) (108,85) (109,77) (110,78) (111,70) (112,61)
(113,76) (114,83) (115,113) (116,74) (117,76) (118,80) (119,75) (120,78)
};

\end{axis}
\end{tikzpicture}
\end{figure}

\begin{figure}[ht]
\caption{The histogram for $q = 1997, n = 4, t = 1000$, starting at $3~2~1~4$. We note that $3$ elements share the same $P$-symbol with $3~2~1~4$, so the behavior here is described by Theorem \ref{thm:limit}. \label{fig:third}}
\begin{tikzpicture}
\begin{axis}[
    ymajorgrids,
    ymin=0, ymax=375,
    xmin=0.5, xmax=24.5,
    xtick=\empty,
    axis x line=bottom,
    axis y line=left,
    width=12cm,
    height=5cm,
    tick style={black},
    every axis plot/.append style={very thick}
]

\addplot+[ycomb, mark=*, mark options={fill=black}, color=black]
coordinates {
(1,0) (2,0) (3,0) (4,0) (5,0) (6,0) (7,0) (8,0)
(9,0) (10,0) (11,0) (12,0) (13,0) (14,0) (15,324) (16,320)
(17,0) (18,356) (19,0) (20,0) (21,0) (22,0) (23,0) (24,0) (17,0) 
};

\addplot+[ycomb, mark=*, mark options={fill=DarkPurple}, color=DarkPurple]
coordinates {
 (15,324) (16,320)
(18,356) 
};

\end{axis}
\end{tikzpicture}
\end{figure}
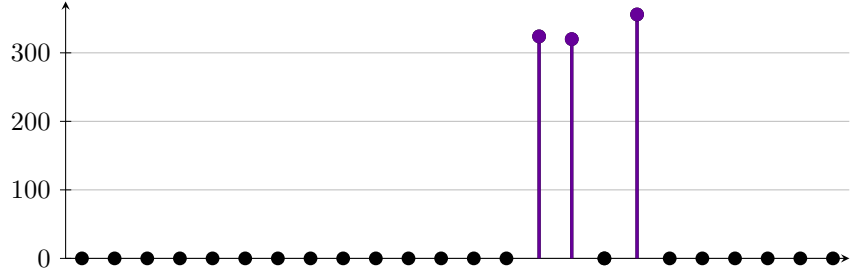

\begin{figure}[ht]
\caption{The histogram for $n = 5, q = 20011, t = 1000$, starting at $3~2~1~4~5$. We note that $4$ elements share the same $P$-symbol with $3~2~1~4~5$ in this case. \label{fig:fourth}}
\begin{tikzpicture}
\begin{axis}[
    ymajorgrids,
    ymin=0, ymax=280, 
    xmin=0.5, xmax=120.5,
    xtick=\empty,
    axis x line=bottom,
    axis y line=left,
    width=12cm,
    height=5cm,
    tick style={black},
    every axis plot/.append style={thin},  
]

\addplot+[ycomb, mark=*, mark options={fill=DarkPurple, scale=0.6}, color=DarkPurple]
coordinates {
 (62,246) (102,259) (103,266)
(105,229)
};

\addplot+[ycomb, mark=*, mark options={fill=black, scale=0.6}, color=black]
coordinates {
(1,0) (2,0) (3,0) (4,0) (5,0) (6,0) (7,0) (8,0)
(9,0) (10,0) (11,0) (12,0) (13,0) (14,0) (15,0) (16,0)
(17,0) (18,0) (19,0) (20,0) (21,0) (22,0) (23,0) (24,0)
(25,0) (26,0) (27,0) (28,0) (29,0) (30,0) (31,0) (32,0)
(33,0) (34,0) (35,0) (36,0) (37,0) (38,0) (39,0) (40,0)
(41,0) (42,0) (43,0) (44,0) (45,0) (46,0) (47,0) (48,0)
(49,0) (50,0) (51,0) (52,0) (53,0) (54,0) (55,0) (56,0)
(57,0) (58,0) (59,0) (60,0) (61,0) (63,0) (64,0)
(65,0) (66,0) (67,0) (68,0) (69,0) (70,0) (71,0) (72,0)
(73,0) (74,0) (75,0) (76,0) (77,0) (78,0) (79,0) (80,0)
(81,0) (82,0) (83,0) (84,0) (85,0) (86,0) (87,0) (88,0)
(89,0) (90,0) (91,0) (92,0) (93,0) (94,0) (95,0) (96,0)
(97,0) (98,0) (99,0) (100,0) (101,0) (104,0)
(106,0) (107,0) (108,0) (109,0) (110,0) (111,0) (112,0)
(113,0) (114,0) (115,0) (116,0) (117,0) (118,0) (119,0) (120,0)
};

\end{axis}
\end{tikzpicture}
\end{figure}

\begin{figure}[ht]
\caption{The histogram for $n = 5, t = 10000, q = 20011$, starting at $3~2~1~4~5$. In this example, the chain happened to visit four different Steinberg cells, staying in each one for a certain number of steps as pictured. The four ``buckets" into which these data cluster are labelled by the pictured Young tableaux according to their color.\label{fig:fifth}}

\begin{tikzpicture}
\begin{axis}[
    ymajorgrids,
    ymin=0, ymax=1300,
    xmin=0.5, xmax=120.5,
    xtick=\empty,
    axis x line=bottom,
    axis y line=left,
    width=12cm,
    height=7cm,
    tick style={black},
    every axis plot/.append style={thin},  
]

\addplot+[ycomb, mark=*, mark options={fill=black, scale=0.6}, color=black]
coordinates {
(1,0) (2,0) (3,0) (4,0) (5,0) (6,0) (7,0) (8,0)
(9,0) (10,0) (11,0) (12,0) (13,0) (14,0) (15,0) (16,0)
(17,0) (18,0) (19,0) (20,0) (21,0) (22,0) (23,0) (24,0) (26,0) (27,0) (28,0) (29,0) (30,0) (32,0)(35,0) (36,0) (37,0) (38,0)(41,0)(43,0) (44,0) (45,0) (46,0) (47,0) (48,0) (49,0) (50,0) (51,0) (52,0) (57,0) (58,0) (59,0) (60,0) (65,0) (66,0) (67,0) (73,0) (74,0) (75,0) (76,0) (68,0) (81,0) (82,0) (83,0) (84,0) (85,0) (86,0) (87,0) (88,0)
(89,0) (90,0) (91,0) (92,0) (93,0) (94,0) (95,0) (96,0)
(97,0) (98,0) (99,0) (100,0) (101,0) (102,0) (103,0) (104,0)(109,0)(112,0)(111,0) (117,0) (118,0) (119,0) (120,0)
(105,0) (106,0) (113,0) (114,0) (115,0) (116,0)  (54,0)  (63,0) (71,0) (72,0) (69,0) (78,0) 
};

\addplot+[ycomb, mark=*, mark options={fill=blue, scale=0.6}, color=blue]
coordinates {
(53,1252)(55,1240) (56,1161)
 (61,1238) (62,1270) (64,1215)

};

\addplot+[ycomb, mark=*, mark options={fill=purple, scale=0.6}, color=purple]
coordinates {
 (70,204) 
 (107,186) (108,190)  (110,190) 

};

\addplot+[ycomb, mark=*, mark options={fill=orange, scale=0.6}, color=orange]
coordinates {
(39,82) (40,78)
 (42,84)
 
(77,91) (79,97) (80,87)

};

\addplot+[ycomb, mark=*, mark options={fill=OliveGreen, scale=0.6}, color=OliveGreen]
coordinates {
(25,327) (31,328) 
(33,345) (34,335)

};

\end{axis}
\end{tikzpicture}

\bigskip

{\color{OliveGreen}\begin{ytableau}
    *(OliveGreen!5)1 & *(OliveGreen!5)3 & *(OliveGreen!5)4 & *(OliveGreen!5)5\\
    *(OliveGreen!5)2
\end{ytableau}} \qquad {\color{orange}\begin{ytableau}
    *(orange!5)1 & *(orange!5)3 & *(orange!5)5\\
    *(orange!5)2\\
    *(orange!5)4 \end{ytableau}}  \qquad  {\color{blue}\begin{ytableau}
    *(blue!5) 1 & *(blue!5)4 & *(blue!5)5\\
    *(blue!5)2 & *(blue!5)3
\end{ytableau}}\qquad   {\color{purple}\begin{ytableau}
    *(purple!5)1 & *(purple!5)4\\
    *(purple!5)2\\
    *(purple!5)3\\
    *(purple!5)5
\end{ytableau}}
\end{figure}
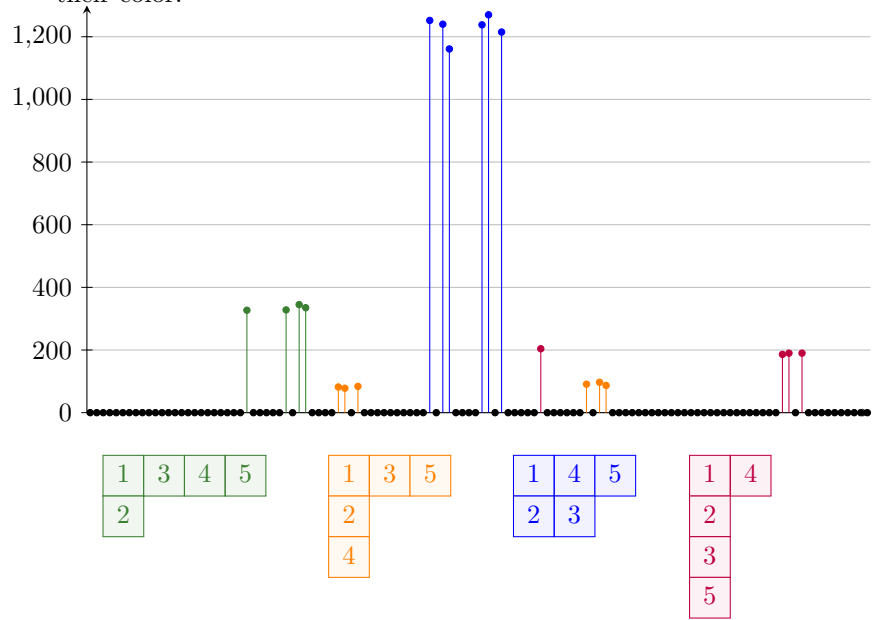


\section{Final comments}
There is much more to say about the Burnside process on $G/B$. First of all, we are interested in explicit algorithms as in Section \ref{sec:sampling} for finite Chevalley groups outside of Type $A$. We are hopeful that a similar approach as in the present paper could accomplish this for classical types, and curious as to whether such an explicit algorithm is possible to describe in any exceptional types (even, for example, for $G_2$). We note that an algorithm implementing the first step of the Burnside process is easy in any type, while the real work is in uniform sampling from Springer fibers.

In any type where such an algorithm can be described, one can then apply Theorem \ref{thm:gentypelimit} as follows. This theorem says that by setting $q$ to be very large, such an algorithm would effectively allow for near-uniform sampling from any Steinberg cell. In \cite{CDH}, it was shown that access to a uniform sampling gives allows for many ways to obtain good estimates for the size of a set. Thus, such an algorithm would provide a direct computational way to answer combinatorial questions about Steinberg cells, e.g.\ how large is a given cell, and how many total cells are there? For example, this idea applied to our algorithm in Type $A$ provides a (redundant) way to estimate the quantities $f^\lambda$, which of course already have precise formulas; this may be less redundant in other types where the combinatorics of Steinberg cells is much more complicated.

In other examples of the Burnside process, there have sometimes been ``interpretable" descriptions for the Markov chain, see e.g.\ the examples in \cite{DH}. We are curious as to whether, even in Type $A$, such an interpretable description could be provided for the Burnside process on $G/B$ studied in the present paper. Further, there is surely more to say about the mixing times for these chains than the lower bound we give in Section \ref{sec:mixing}, and a more refined estimate would be interesting.

\clearpage
\bibliographystyle{plainurl}
\bibliography{bibl}

\end{document}